\documentclass[a4paper,12pt]{amsart}
\usepackage{amsmath,amssymb,amsthm}
\usepackage{etoolbox}
\usepackage{url}
\usepackage{graphicx}
\usepackage{enumitem}
\usepackage{xcolor,colortbl}
\usepackage{mathtools}
\usepackage{enumitem}
\definecolor{refkey}{rgb}{0,0,1}
\definecolor{labelkey}{rgb}{0,0,1}
\usepackage{soul}
\usepackage{hyperref}
\usepackage{tabularx,booktabs}

\usepackage[margin=1.1in]{geometry}

\newtheorem{lemma}{Lemma}[section]
\newtheorem{theorem}[lemma]{Theorem}
\newtheorem{corollary}[lemma]{Corollary}
\newtheorem{proposition}[lemma]{Proposition}

\theoremstyle{remark}

\newtheorem{example}[lemma]{Example}
\AtEndEnvironment{example}{\null\hfill\qedsymbol}

\newcommand{\gp}{\mathsc{gp}}
\newcommand{\barf}{{f^*}}
\newcommand{\hatf}{{f^\star}}
\newcommand{\barF}{{F^*}}
\newcommand{\hatF}{{F^\star}}

\newcommand{\edgetree}{\mathcal{E}}
\newcommand{\archset}[1]{\mathcal{A}_{#1}}
\newcommand{\dualset}[1]{\mathcal{D}_{#1}}
\newcommand{\sbe}{\mathsc{sb}}

\newcommand{\map}[3]{{#1}:{#2}\rightarrow{#3}}

\newcommand{\ninv}{\mathrm{coinv}}
\newcommand{\ninvk}{\ninv_k}

\newcommand{\inv}{\mathrm{inv}}
\newcommand{\invk}{\inv_k}

\newcommand{\R}{\mathcal{R}}
\newcommand{\T}{\mathcal{T}}
\newcommand{\F}{\mathcal{F}}

\newcommand{\C}{\mathcal{C}}

\newcommand{\Sym}[1]{\mathfrak{S}_{#1}}

\newcommand{\area}{\mathrm{area}}
\newcommand{\coarea}{\mathrm{coarea}}
\newcommand{\p}{\mathcal{P}}

\newcommand{\barfset}{\F^*}
\newcommand{\hatfset}{{\F}^\star}
\newcommand{\disp}{\mathrm{disp}}

\newcommand{\depth}{\mathrm{dep}}

\newcommand{\chr}{\mathrm{chr}}
\newcommand{\cochr}{{\mathrm{cochr}}}
\newcommand{\col}{\kappa}

\newcommand{\he}{h}
\newcommand{\majle}{\mathrm{maj}}
\newcommand{\majre}{\mathrm{comaj}}

\newcommand{\mathsc}[1]{{\normalfont\textsc{#1}}}

\newcommand{\factmap}{\mathsc{fact}}

\newcommand{\arch}[1]{\mathsc{arch}(#1)}
\newcommand{\archmap}{\mathsc{arch}}
\newcommand{\dual}[1]{\mathsc{dual}(#1)}
\newcommand{\dualmap}{\mathsc{dual}}
\newcommand{\cay}[1]{\mathsc{cay}(#1)}
\newcommand{\caymap}{\mathsc{cay}}
\newcommand{\brk}[1]{\mathsc{lower}(#1)}
\newcommand{\brkmap}{\mathsc{lower}}
\newcommand{\hatbrk}[1]{\mathsc{upper}(#1)}
\newcommand{\hatbrkmap}{\mathsc{upper}}
\newcommand{\cda}[1]{\mathsc{cda}(#1)}
\newcommand{\cdamap}{\mathsc{cda}}

\newcommand{\expfuncmap}{\mathsc{expand}}
\newcommand{\jcdab}[1]{\mathsc{jcdal}(#1)}

\newcommand{\jcdabmap}{\mathsc{jcdal}}

\newcommand{\join}[1]{\mathsc{join}(#1)}

\newcommand{\joinmap}{\mathsc{join}}

\newcommand{\arealk}{\area_k}
\newcommand{\areauk}{\coarea_k}
\newcommand{\areal}{\area}
\newcommand{\areau}{\coarea}
\newcommand{\majl}{\mathrm{maj}}
\newcommand{\majr}{\mathrm{comaj}}
\newcommand{\majlk}{\mathrm{maj}_k}
\newcommand{\majrk}{\mathrm{comaj}_k}
\newcommand{\sarealk}{\mathrm{semiarea}_k}
\newcommand{\sareauk}{\mathrm{cosemiarea}_k}
\newcommand{\genkfact}[2]{(a_1^0\, \cdots\, a_1^{#1}) \cdots (a_{#2}^0\,
\cdots\, a_{#2}^{#1})}
\newcommand{\barroot}{\R^*}
\newcommand{\hatroot}{{\R}^{\star}}

\newcommand{\Fbar}{F^*}
\newcommand{\spann}{\mathrm{span}}

\newcommand{\s}{\sigma}

\newcommand{\tfact}{\tau}

\author{John Irving}
\address{Saint Mary's University, Halifax, NS, Canada}
\email{john.irving@smu.ca}
\author{Amarpreet Rattan}
\address{Department of Mathematics, Simon Fraser University, Burnaby, BC, Canada}
\email{rattan@sfu.ca}
\title[$k$-factorizations and Mahonian statistics on $k$-forests]{$k$-Factorizations of the full cycle and generalized Mahonian statistics on $k$-forests}

\begin{document}

\newcommand{\exclude}[1]{}
\newcommand{\TODO}[1]{\noindent{\color{blue} {\scshape To Do:} #1}}

\dedicatory{Dedicated to Ian Goulden and David Jackson on the occasion of
	their retirement.}

\begin{abstract}
We develop direct bijections between the set  $\F_n^k$ of minimal factorizations of the long cycle $(0\,1\,\cdots\, kn)$ into $(k+1)$-cycle factors and the set $\R_n^k$ of rooted labelled forests on vertices $\{1,\ldots,n\}$ with  edges coloured with $\{0,1,\ldots,k-1\}$ that map natural statistics on the former to generalized Mahonian statistics on the latter. In particular, we examine the generalized \emph{major index} on forests $\R_n^k$ and show that it has a simple and natural interpretation in the context of factorizations. Our results extend those in~\cite{irvrat}, which treated the case $k=1$ through a different approach, and provide a bijective proof of the equidistribution observed by Yan  \cite{yank} between displacement of $k$-parking functions and generalized inversions of $k$-forests. \end{abstract}

\maketitle

\section{Introduction}\label{sec:intro}

The aim of this article is to recast and generalize our earlier work~\cite{irvrat} concerning connections between rooted  forests, parking functions, and factorizations of cycles into transpositions. We begin by briefly reviewing  these   objects and the main result of~\cite{irvrat}.  Novel content begins in Section~\ref{sec:facts}.

The following notational conventions are used throughout.
For nonnegative integers $m \leq n$, let $[n] : = \{0, \dots,
n\}$ and $[m, n] = \{m, \ldots, n\}$.  
The symmetric group on $X \subseteq [n]$ is denoted $\Sym{X}$. Permutations $\s,\tau \in \Sym{X}$ are multiplied \emph{left to right}, and cycles in $\Sym{X}$ are always presented with least element first; i.e. in the form $(a_0\, a_1\, \cdots\, a_m)$ with $a_0 = \min_i a_i$    The canonical full cycle $(0\, 1\, 2\, \cdots\, n) \in \Sym{[n]}$ will be denoted $\s_n$.

\subsection{Mahonian Statistics on Rooted Forests}\label{sec:forests}

A \emph{rooted forest} is graph whose components are rooted trees, \emph{i.e.} trees with a distinguished vertex.
Let $\R_n$ be the set of rooted forests on vertices $[1,n]$.

For convenience we regard the edges of every  forest $F \in \R_n$  as being directed away from the roots of their components.  We identify an edge directed from $u$ to $v$ by the pair $(u,v)$.  If $F$ contains such an edge then we say $u$ is  the \emph{parent} of $v$ and $v$ is  a \emph{child} of $u$.  More generally,  $u$ is  an \emph{ancestor} of  $v$ --- and $v$ is a \emph{descendant} of $u$ --- if there is a nonempty directed path from $u$ to $v$.   The subtree of $F$ induced by $u$ and all its descendants is called the \emph{hook} at $u$.  We write $H(u)$ for this hook and $h(u)$ for  the number of vertices contained therein, commonly known as the \emph{hook length} at $u$.  The \emph{total depth} of $F$ is the sum of all non-root hook lengths,
$$
	\depth(F) := \sum_{(u,v) \in E(F)} h(v).
$$
Equivalently, this is the sum of the distances from all verties to the roots of their components, sometimes known as the  \emph{path length} of $F$.

The \emph{major} and \emph{comajor indices} of $F$ are defined by
\begin{equation*}
	\majl(F) := \sum_{u \in V(F)} h_L(u)\qquad \textnormal{ and
		}\qquad \majr(F) := \sum_{u \in V(F)} h_R(u),
\end{equation*}
where  
\begin{equation*}\label{eq:lrhook}
	h_{L}(u) := \sum_{\substack{(u,v) \in E(F) \\ v < u}} h(v)\qquad \textnormal{ and }\qquad h_R(u) := \sum_{\substack{(u,v) \in E(F) \\ u < v}} h(v).
\end{equation*}
We refer to the quantities $h_L(u)$ and $h_R(u)$ as the \emph{left} and \emph{right hook lengths} at $u$.   The rationale for this terminology will be apparent later.  
Note that  $h(u) = h_L(u) + h_R(u) + 1$ and thus
\begin{equation*}
	\majl(F) + \majr(F) = 
	\sum_{u \in V(F)} (h(u) - 1) =  \depth(F).
\end{equation*}

If $v$ is a descendant of $u$,  then the pair $(u,v)$ is said to be an \emph{inversion} of $F$ when $u > v$ and a \emph{coinversion} when $u < v$.  Let $\inv(u) := \#\{v \in H(u) \,:\, v < u\}$ and $\ninv(u) := \#\{v \in H(u) \,:\, v > u\}$ denote the  number of inversions and coinversions in $F$ of the form $(u,v)$ for some $v$.   Clearly $h(u)=\inv(u)+\ninv(u)+1$.
The \emph{inversion} and \emph{coinversion indices} of $F$ are defined by    
\begin{equation*}
	\inv(F) := \sum_{u \in V(F)} \inv(u) \qquad \textnormal{ and
		}\qquad \ninv(F) := \sum_{u \in V(F)} \ninv(u).
\end{equation*}
These are simply the total number of inversions and coinversions in $F$. Observe that $\inv(F) + \ninv(F)=\depth(F)$, because every pair of vertices
$(u,v)$ with $v \in H(u)$ is either an inversion or a
coinversion, but not both. 

Figure~\ref{fig:mahonian} shows a rooted forest along with several statistics.
\begin{figure}[tbp]
	\centering
	\includegraphics[width=.9\textwidth]{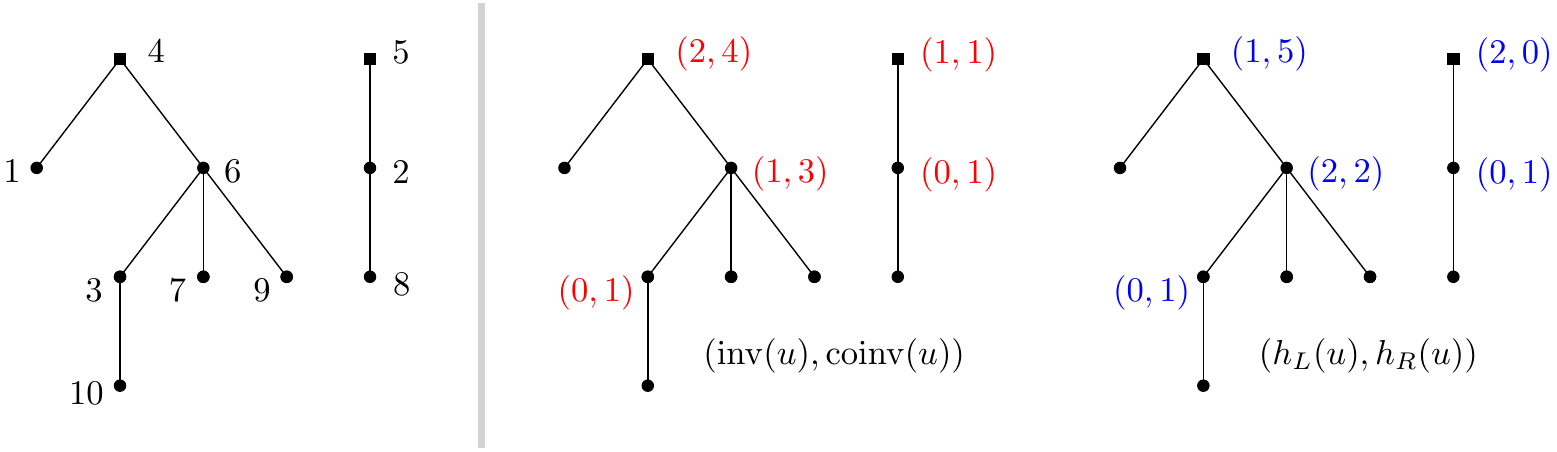}
	\caption{A forest $F \in \R_{10}$ (left) with $\depth(F)=14$, $\inv(F)=4$, $\ninv(F)=10$, $\majl(F)=5$ and $\majr(F)=9$. Statistics $(\inv(u),\ninv(u))$ and $(h_L(u),h_R(u))$ are shown beside non-leaf vertices $u$ in red and blue (right). }
	\label{fig:mahonian}
\end{figure}

The inversion/major indices on forests are generalizations of  well-known {Mahonian statistics} on $\Sym{n}$ of the same name. The extensions of $\inv(\cdot)$ and $\majl(\cdot)$ from $\Sym{n}$ to $\R_n$ are due to Mallows and Riordan \cite{mallowsriordan} and Bj\"orner and Wachs \cite{bjorn}, respectively.

Note that $\inv(F)$ and $\ninv(F)$ are equidistributed over $\R_n$, as can be seen by exchanging vertex labels $i$ and $n+1-i$. The same is true of $\majl(F)$ and $\majr(F)$.  Our interest lies in the joint distributions $(\inv,\ninv)$ and $(\majl,\majr)$, which turn out to coincide over $\R_n$ just as they do over $\Sym{n}$.\footnote{The well-known joint symmetry of $(\inv,\majl)$ over $\Sym{n}$ does \emph{not} extend to $\R_n$.} We will elaborate  on the relationship between these statistics in Section \ref{sec:main}.  

Let $\T_n$ be the set of  trees on vertices $[n]$. Note that removal of vertex 0 puts these trees in natural correspondence with rooted forests on $[1,n]$.  While we have cast our work in terms of rooted forests, all statements regarding $\R_n$ can be translated \emph{mutatis mutandis} to the language of trees. 

\subsection{Factorizations of Full Cycles}
\label{sec:factorizations}
\newcommand{\cycles}[2]{\C^{#1}_{#2}}

It is well known that every $r$-cycle $\s \in \Sym{[n]}$ can be expressed as a product of $r-1$ transpositions but no fewer.  Accordingly,  
a sequence $(\tau_1, \ldots,\tau_{r-1})$ of transpositions  satisfying  $\s=\tau_1\tau_2\cdots\tau_{r-1}$ is called a \emph{minimal factorization} of $\s$. For example, the canonical decompositions   
\begin{align}
\label{eq:expdecomp}
 (a_0\,a_1)(a_0\,a_2)\cdots(a_0\,a_r)
\end{align}
and
\begin{align}
\label{eq:upperdecomp}
 (a_0\,a_r)(a_1\,a_r)\cdots(a_{r-1}\,a_r)
\end{align}
are minimal factorizations of $\s=(a_0\, a_1\, \cdots\, a_r)$.  These will play  a central role in our analysis and we refer to them as the \emph{lower} and \emph{upper decompositions} of $\s$, respectively.

Let $\F_n$ be the set of minimal factorizations of the  full cycle $\s_n=(0\,1\,2\,\cdots\,n)$.  For example
\begin{align*}
	\F_{1} &= \{(0\; 1)\},\\
	\F_{2} &= \{(0\; 1) (0 \; 2), (0\; 2) (1\; 2), (1\; 2)
	(0\; 1)\}.
\end{align*}  
Minimal factorizations of a fixed full cycle have long been known to be  related to labelled trees (equivalently, rooted forests).   The identity $|\F_n|=(n+1)^{n-1}=|\R_n|$  dates back at least to Hurwitz but is often credited to D\'enes~\cite{denes}, who offered an elegant proof via indirect counting. Direct bijections between $\F_n$ and $\R_n$  came later.  The simplest of these, due to Moszkowski~\cite{mosz}, has been rediscovered in different guises by a number of authors. Its essence is the fact that trees are a special class of planar maps, and  minimal factorizations (broadly speaking) serve as combinatorial encodings of planar embeddings. A version of this bijection is central to our construction and described in Section~\ref{sec:arch}.

The connection between $\R_n$ and $\F_n$ can be refined  to account for  forest inversions/coinversions.  The corresponding statistics for factorizations, which we call \emph{area} and \emph{coarea}, are defined for $f=(a_1\, b_1)(a_2\, b_2)\cdots(a_n\, b_n) \in \F_n$ by
\begin{equation*}\label{def:areacoarea}
	\areal(f) =  \binom{n}{2} - \sum_{i=1}^n a_i
	\qquad \textnormal{ and }\qquad
  \areau(f) =  \sum_{i=1}^n (b_i-1) - \binom{n}{2}.
\end{equation*} 
With this terminology the main result of our previous paper~\cite{irvrat}  can be stated as follows.\footnote{In  \cite{irvrat} this result is phrased in terms of $\T_n$, and area and coarea are called \emph{lower} and \emph{upper area}, respectively.}	

\begin{theorem}
	For any $n \geq 0$, the bi-statistics $(\inv, \ninv)$ on $\R_n$ and $(\areal,
	\areau)$ on $\F_n$ share the same joint distribution.\label{thm:oldirvrat}
\end{theorem}
 
 Our proof of Theorem~\ref{thm:oldirvrat} relied on generating series techniques but was effectively based on a recursive bijection. In Section~\ref{sec:kone} we shall reestablish this result by describing a natural and \emph{direct} bijection between $\F_n$ and $\R_n$ that maps $(\areal,\areau)$ to $(\majl,\majr)$.   This will serve as a base case toward  extending Theorem~\ref{thm:oldirvrat} to treat factorizations into $(k+1)$-cycles for arbitrary $k \geq 1$.  
 
\subsection{Minimal $k$-Factorizations}\label{sec:facts}

For $k,n \geq 1$, let  $\F_n^k$ be the set of all sequences $(\tfact_1,\ldots,\tfact_n)$ of $(k+1)$-cycles $\tfact_i \in \Sym{[kn]}$  such that $\tfact_1\tfact_2\cdots\tfact_n=\s_{kn}$.  In particular, we have $\F_n^1=\F_n$.

Certainly $\F_n^k$  is nonempty, as taking $\tfact_i=(0, in+1, in+2, \ldots,in)$  for $1 \leq k \leq n$ defines one canonical element.   Moreover, the cycle $\s_{kn}$ cannot be factored into fewer $(k+1)$-cycles, since  replacing each factor with its lower expansion would then yield a factorization into fewer than $kn$ transpositions. As such we call the elements of $\F_n^k$  \emph{minimal $k$-factorizations} of $\s_{kn}$, or simply \emph{$k$-factorizations} for short.   

Minimal $k$-factorizations of full cycles are well-studied and, unsurprisingly,  they  correspond with a class of decorated forests. These will be defined shortly,
but in preparation for stating our main result, we  first describe how the area/coarea statistics on $\F_n$ can be extended to $\F_n^k$. 

 Let
\begin{equation}\label{eq:factform}
	f = \genkfact{k}{n}
\end{equation}
be a generic element of $\F_n^k$, keeping in mind our convention that $a_i^0$ is the least element of the $i$-th factor. Then we define
\begin{equation*}\label{def:areas}
	\begin{split}
	\arealk(f) &:=  \binom{kn}{2} - n\binom{k}{2}- k \sum_{j=1}^n a^0_j,\\
	 \areauk(f) &:= k \sum_{j=1}^n (a^{k}_j -
	1) - \binom{kn}{2} - n\binom{k}{2}.
\end{split}
\end{equation*}
The careful reader may observe that  these expressions share a  common factor of $k$, \emph{e.g.} $\arealk(f)=k\cdot (k\binom{n}{2}-\sum_j a_j^0)$.  This apparent redundancy will arise naturally in our analysis so we have chosen not to ``normalize'' it out of our definitions.  

 We further introduce  two additional statistics on $\F_n^k$ that we call  \emph{semiarea} and
\emph{cosemiarea}, given by 
\begin{equation*}\label{eq:defsemis}
	\begin{split}
	\sarealk(f) &:= \binom{kn}{2} - \sum_{i=0}^{k-1} \sum_{j=1}^n
	a_j^i,\\
	 \sareauk(f) &:= \sum_{i=1}^{k}
	\sum_{j=1}^n (a_j^i - 1) - \binom{kn}{2}.
\end{split}
\end{equation*}
Although not obvious from these definitions, it transpires that $\sarealk(f)$ and $\sareauk(f)$ are also always  divisible by $k$.
Note that {both} $\arealk(f)$ and $\sarealk(f)$ revert to $\areal(f)$ at $k=1$, but are otherwise distinct.  The same is true for coareas.


\subsection{Rooted $k$-Forests} \label{sec:kforests}
A \emph{$k$-forest}  is a rooted labelled forest $F$ equipped with a function $\col
: E(F) \rightarrow [k-1]$ that assigns one of $k$  \emph{colours} $\{0,1,\ldots,k-1\}$ to each edge of $F$. Two $k$-forests $(F,\col)$ and $(F',\col')$ are isomorphic if there is a graph isomorphism $\zeta\,:\,V(F)\rightarrow V(F')$ that preserves roots  and edge colours.

  Let $\R_n^k$  be the set of  $k$-forests on  vertices $[1,n]$.   Note that there are no  restrictions on the colouring function, so there are $k^{|E(F)|}=k^{n-c}$ elements of $\R_n^k$ arising from any  forest $F \in \R_n$ with $c$ components.   For brevity we will suppress explicit mention of the colouring function when working with elements of $\R_n^k$.

Our definition of $k$-forests follows Yan~\cite{yank}, who attributes them to Stanley. These objects appear elsewhere in the literature in the different but equivalent form of \emph{$k$-cacti}, which are tree-like structures with edges replaced by $(k+1)$-gons.  
In particular, there is a simple correspondence between $k$-forests on $n$ labelled vertices and $k$-cacti with $n$ labelled polygons.   See Figure~\ref{fig:kforests}.
\begin{figure}[tbp]
	\centering
	\includegraphics[width=12cm]{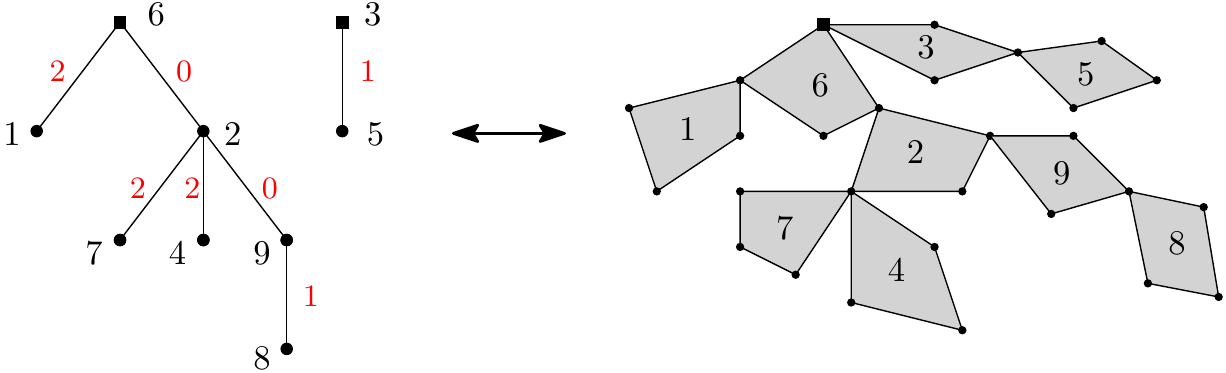}
	\caption{A $3$-forest $F$ and its associated $3$-cactus $C$.  Vertices of of $F$ correspond with polygons of $C$. The colouring of $F$ is indicated in red. The colour  of edge $(i,j)$ describes where $P_j$ (polygon $j$) is attached to $P_i$ in $C$. For example, edge $(9,8)$ with colour 1 indicates $P_8$ is attached to $P_9$ at the second vertex clockwise from where $P_9$ is attached to its parent, $P_2$.}
	\label{fig:kforests}
\end{figure}

A straightforward extension of the Pr\"ufer encoding yields  $|\R_n^k|=(kn+1)^{n-1}$, and the Moszkowski correspondence $\F_n \leftrightarrow \R_n$ likewise extends to a bijection between $\F_n^k$ and $\R_n^k$. In particular, we have $|\F_n^k| =(kn+1)^{n-1}$.  We direct the reader to \cite{irving1} and  references therein for  further information, however all details relevant to our discussion will be outlined when required.  

The Mahonian statistics $\inv/\ninv$ and $\majl/\majr$ on $\R_n$ can be applied to $\R_n^k$ simply by ignoring  edge colourings. However, we also wish to introduce extended versions of these statistics specifically for $k$-forests.

To this end we first define the
\emph{chromatic depth} and \emph{chromatic codepth} of $F \in \R_n^k$ by
\begin{equation}
	\chr_k(F) = \sum_{e=(u,v) \in E(F)} \col(e) \cdot
	h(v)\;\;\; \textnormal{and}\;\;\; \cochr_k(F) = \sum_{e=(u,v) \in E(F)} 
(k-1 -
	\col(e)) \cdot
	h(v),
	\label{eq:cdepth} 
\end{equation}
respectively. Note that 
$\chr_k(F) = \depth(F)$ when $F$ has all its edges coloured 1, so  $\chr_k(F)$ 
can be regarded as a colour-weighted depth of $F$.  We then let
\begin{equation*}
\begin{split}
		\invk(F) &:= \inv(F) + \chr_k(F),\\
		\ninvk(F) &:= \ninv(F) + \cochr_k(F),
\end{split}
\end{equation*}
and
\begin{equation*}
\begin{split}
	\majlk(F) &:= \majl(F) + \chr_k(F), \\
	 \majrk(F) &:= \majr(F) + \cochr_k(F).
\end{split}
\end{equation*}
These extended statistics revert to $\inv/\majl$ at $k=1$ since $\chr_1(F)=\cochr_1(F)=0$ for all  $F \in \R_n$.

This definition of $\invk(F)$ is equivalent to that appearing in \cite{yank}, 
where it arose in an effort to  generalize Kreweras' identity  between the 
inversion enumerator for trees and the discrepancy enumerator for parking 
functions (see \cite{kreweras}).  This connection will be discussed in more 
detail in Section~\ref{sec:parks}. Consideration of $\majlk(F)$ appears to be 
novel.   It is straightforward (using Figure \ref{fig:kforests} as a guide) to
adapt the definitions of $\invk, \ninvk, \majlk,$ and $\majrk$ so they apply to $k$-cacti.  

\subsection{Main Result}
\label{sec:main}

We can now state our main theorem, which relates area statistics on $k$-factorizations to  major
indices on $k$-forests.  
\exclude{
\begin{theorem}
	Let $n$ and $k$ be positive integers.  Then the map
	$\jcdabmap: \F_n^k \rightarrow \R_n^k$ is a bijection.
	Furthermore, if $f \in \F_n^k$ and $F = \jcdab{f}$,  then
	\begin{align}
	\begin{split}
		\arealk(f) &= k \cdot \majlk(F),\\
		\sareauk(f) &= k \cdot \majr(F),
	\end{split}
\\
	\begin{split}
		\sarealk(f) &= k \cdot \majl(F),\\
		\textnormal{ and }\,\,\,\, \areauk(f) &= k \cdot \majrk(F).
	\end{split}
	\end{align}
	\label{thm:maintheorem}
\end{theorem}
}
\begin{theorem}
	\label{thm:maintheorem}
	For any  $n,k \geq 1$, there is an explicit bijection
	$\phi: \F_n^k \rightarrow \R_n^k$ such that for $F=\phi(f)$ we have
		\begin{xalignat*}{2}
		\arealk(f) &= k \cdot \majlk(F), & 
		\areauk(f) &= k \cdot \majrk(F) \\
		\intertext{and}	
		\sarealk(f) &= k \cdot \majl(F), &
		\sareauk(f) &= k \cdot \majr(F). 
	\end{xalignat*}
	Moreover, if $f=\genkfact{k}{n}$ then $a_i^k-a_i^0=k\cdot h(i)$ for $i \in [1,n]$, where $h(i)$ is the  hook length at vertex $i$ in $F$.\end{theorem}

Theorem~\ref{thm:maintheorem} will be proved first for $k=1$ in Section~\ref{sec:kone} and then in general in Section~\ref{sec:genk}. In each case the proof is through  construction of the promised bijection $\phi$. 
Note that the theorem implies  $\sarealk(f)$ and $\sareauk(f)$ are independent of the edge colouring of  $F=\phi(f)$. The edge colours of $F$ are only relevant to the evaluation of $\arealk(f)$ and $\areauk(f)$.

We will now describe precisely how Theorem~\ref{thm:maintheorem} can be viewed 
as a generalization of Theorem~\ref{thm:oldirvrat}.  The key is that $\majl(F)$ 
and $\inv(F)$ are equidistributed not only over $\R_n$, but over every 
isomorphism class thereof. This was first proved inductively by Bj\"orner and 
Wachs~\cite{bjorn}, and later bijectively by Liang and Wachs~\cite{liangwachs}.  
More recently, Grady and Poznanovi\'{c} \cite{grady-poz}  established this 
result by mapping both $\inv(F)$ and $\majl(F)$ to a common  code  called a 
\emph{subexcedant sequence} on $F$.  We will not go into further detail here. 
The salient point is that there are known bijections  $\xi\,:\,\R_n \rightarrow 
\R_n$  satisfying $\majl(F)=\inv(\xi(F))$ with $F$ isomorphic to $\xi(F)$ 
for all $F \in \R_n$.  For definiteness, let $\gp$ be the  Grady-Poznanovi\'{c} 
bijection of this type.  

Note that  $\gp$ extends to a bijection on $\R_n^k$ by effectively ignoring edge colours.
\exclude{That is, $(F,\col) \mapsto (\gp(F),\col \circ \gp)$.}  
Let $F \in \R_n^k$ and $F'=\gp(F)$, so that $F$ and $F'$ are isomorphic as $k$-forests and $\majl(F)=\inv(F')$. Then
\begin{align*}
\depth(F)=\depth(F') 
&\implies \majl(F) + \majr(F) = \inv(F')+\ninv(F')  \\
&\implies \majr(F)=\ninv(F').
\end{align*}
It is clear from~\eqref{eq:cdepth} that chromatic depth and codepth are invariant on isomorphism classes of $\R_n^k$, so there follows $\majlk(F)=\invk(F')$ and $\majrk(F)=\ninvk(F')$.

Let $\phi: \F_n^k \rightarrow \R_n^k$ be the bijection guaranteed by Theorem~\ref{thm:maintheorem}.  Then taking $\hat{\phi}=\gp \circ \phi$ proves the following generalization of Theorem~\ref{thm:oldirvrat}.   We shall see in the next section how this sheds light on an open question concerning the relationship between $k$-forests and generalized parking functions.

\begin{corollary}
\label{cor:maincor}
	For any $n,k \geq 1$, there is an explicit bijection $\hat{\phi}: \F_n^k \rightarrow \R_n^k$ such that for $F=\hat{\phi}(f)$ we have $(\arealk(f),\areauk(f)) = k \cdot (\invk(F), \ninvk(F))$ and $(\sarealk(f),\sareauk(f)) = k \cdot (\inv(F), \ninv(F))$.
	\label{thm:maincor}
\end{corollary}

Let us now consider the latter claim of Theorem~\ref{thm:maintheorem}, regarding hook lengths. In case $k=1$ the theorem stipulates that the distribution of the hook length vector $(h(1),\ldots,h(n))$ over $F \in \R_n$ matches that of the \emph{difference index} $(b_1-a_1,\cdots,b_n-a_n)$ over $f =(a_1\,b_1)\cdots(a_n\,b_n) \in \F_n$.  A similar result appears in~\cite{gouldenyong}, although there the authors compute transposition differences circularly, replacing $b-a$ with $\min\{b-a, n+1+a-b\}$.  This has the effect of disguising the connection with hook lengths, despite the use of a dual construction equivalent to that used here (see Section~\ref{sec:dual}).

Over the past couple decades, a considerable amount of effort has been put toward the development of  \emph{hook length formulae} for trees and forests.  These formulae  generally provide simple multiplicative expressions for sums of the form $\sum_{T \in \mathcal{T}} \prod_{v \in V(T)} \alpha(v)$, where $\mathcal{T}$ is a class of rooted trees and $\alpha(v)$ can be expressed in terms of the hook length of $T$ at $v$.  

\newcommand{\Aut}[1]{\mathrm{Aut}(#1)}
For the class of rooted labelled forests, one of the simplest such formulae is 
\begin{equation}
\label{eq:hookformula}
	\sum_{F \in \R_n} z^{c(F)} \prod_{v \in V(F)} \frac{1}{h(v)} = z(z+1)(z+2)\cdots(z+n-1),
\end{equation}
where $c(F)$ is the number of components in $F$.
This reflects the fact that permutations on $[1,n]$ with $m$ cycles --- which are well known to be counted by the coefficient of $z^m$ on the right-hand side, \emph{i.e.} the signless Stirling number $s(n,m)$ --- are in correspondence with increasing rooted forests on  $[1,n]$ with $m$ components.  See~\cite{gesselseo} for a more general approach, in particular Corollary 6.3, of which~\eqref{eq:hookformula} is a special case.

Using~\eqref{eq:hookformula} at $z=1$ together with Theorem~\ref{thm:maintheorem} at $k=1$ yields the curious identity
$$	
\sum_{f \in \F_n} \prod_{i=1}^n \frac{1}{b_i-a_i} = n!,
$$
where the sum extends over all factorizations $f =(a_1\,b_1)\cdots(a_{n}\,b_{n}) \in \F_n$.   More generally, since there are $k^{n-c(F)}$ ways of colouring a forest $F \in \R_n$ to create a $k$-forest $F' \in \R_n^k$, we can apply~\eqref{eq:hookformula} at $z=1/k$ to get:

\begin{corollary}
For any $n,k \geq 1$, we have
\begin{align*}
\sum_{f \in \F_n^k} \prod_{i=1}^n \frac{1}{a_i^k-a_i^0} 
	= 	\frac{(k+1)(2k+1)\cdots((n-1)k+1)}{k^n},
\end{align*}
where the sum extends over all  $f =\genkfact{k}{n} \in \F_n^k$.
\end{corollary}

\subsection{$k$-Parking Functions}\label{sec:parks}

A sequence $p=(a_1, \dots, a_n)$ of
nonnegative integers is called a \emph{$k$-parking function}
if its nondecreasing rearrangement $(b_1, \dots, b_n)$  satisfies $b_i \leq k(i-1)$ for $1 \leq i \leq n$.  Let $\p_n^k$ be the set of $k$-parking functions of length $n$.  
Elements of $\p_n:=\p_n^1$  are simply called \emph{parking
functions}.   There is an extensive body of literature on these object and we will only skim the surface here. We refer the reader to the  comprehensive surveys by Yan \cite{yanhandbook} and Haglund \cite{hag} for further information.

It is well known, and easy to prove via cycle lemma or direct bijection, that $|\p_n^k|=|\R_n^k|=(kn+1)^{n-1}$.  This  can be refined to account for inversions in $k$-forests. The companion statistic  on $k$-parking functions is called \emph{displacement}, defined for $p =(a_1,\ldots,a_n) \in \p_n^k$  by
\begin{equation}
\label{eq:dispdefn}
	\disp_k(p) := k \binom{n}{2} - \sum_{i=1}^n a_i.
\end{equation}
Then we have the following result,  which was first proved for $k=1$ by Kreweras
\cite{kreweras} and then for general $k \geq 1$ by Yan.
\begin{theorem}[\cite{yank}]
	For $n, k \geq 1$ 
		\begin{equation*}\label{eq:yank}
		\sum_{p \in \p_n^k} q^{\disp_k(p)}
=
		\sum_{F \in \R_n^k} q^{\invk(F)}.
	\end{equation*}
	\label{thm:kparkinv}
\end{theorem}

Yan's proof of Theorem \ref{thm:kparkinv} is inductive, and it has been an open problem to find a bijective proof. 
Such a proof is afforded by Corollary~\ref{cor:maincor} in tandem  with the simple correspondence between $\F_n^k$ and $\p_n^k$  described below.

\begin{proposition}
For fixed $n,k \geq 1$, define $L \,:\, \F_n^k \rightarrow [kn]^n$ by 
$$
L:\genkfact{k}{n} \mapsto (a_1^0, \dots, a_n^0).
$$ 
Then $L$ is a bijection from $\F_n^k$ to $\p_n^k$, and $\area_k(f)=k\cdot \disp_k(L(f))$ for all $f$.
	\label{thm:factstopark}
\end{proposition}

 Assuming  the truth of the proposition, and letting $\hat{\phi}\,:\,\F_n^k \rightarrow \R_n^k$ be the bijection from Corollary~\ref{cor:maincor}, observe that  $L \circ \hat{\phi}^{-1}$  maps $\R_n^k$ bijectively to $\p_n^k$ while satisfying  $\invk(F)=\frac{1}{k}\cdot \areal_k(\hat{\phi}^{-1}(f))=\disp_k(L(\hat{\phi}^{-1}(F)))$.  This is the promised bijective proof of Theorem~\ref{thm:kparkinv}. 

The latter claim of Proposition~\ref{thm:factstopark} is an immediate 
consequence of the definitions. 
The bijectivity of $\map{L}{\F_n^k}{\p_n^k}$ when $k=1$ was proved
explicitly by Biane~\cite{biane} and is equivalent to an earlier result of Stanley~\cite[Theorem 3.1]{stanleyparking}. We  call this special case of $L$ the \emph{Stanley-Biane bijection}, denoted $\sbe : \F_n \rightarrow \p_n$.  A proof of   bijectivity for general $k$ recently appeared in \cite{williams}, where it was established through a straightforward generalization of Biane's argument.  We will now describe a different (independent) proof that relies on a  reduction to $\mathsc{sb}$.  

Consider the function $\expfuncmap : \p_n^k \rightarrow \p_{kn}$ that replaces each entry of a $k$-parking function with $k$ copies of itself.  For instance, $p=(0,5,1) \in \p_3^3$ has $\expfuncmap(p)=(0,0,0,5,5,5,1,1,1) \in \p_9$. 

Similarly define the function $\brkmap : \F_n^k \rightarrow
\F_{kn}$ that replaces each cyclic factor of a $k$-factorization with its lower decomposition; see~\eqref{eq:expdecomp}. For example, $f=(1\,2\,5)(0\,1\,6)(3\,4\,5) \in \F_3^2$ has $\brkmap(f)=(1\,2)(1\,5)(0\,1)(0\,6)(3\,4)(3\,5) \in \F_{6}$.  

Certainly both $\expfuncmap$ and $\brkmap$ are injective for all $k$ (although not bijective for $k > 1$). It is also clear that $\sbe \circ \brkmap(\F_n^k) \subseteq \expfuncmap(\p_n^k)$.  Thus $\expfuncmap^{-1} \circ \sbe \circ
\brkmap$ is a well-defined injective function from $\F_n^k$ to $\p_n^k$, and by definition it agrees with $L$ on its domain.  Since  $|\F_n^k|=|\p_n^k|$,  we conclude that $L$ is bijective.

Finally, a comment on nomenclature. It is common to view a $k$-parking function $p=(a_1,\ldots,a_n) \in \p_n^k$ as a specially labelled  lattice path  from $(0,0)$ to $(n, kn)$, with unit steps to the north and east, that remains weakly below the line $y=kx$. For each  $h \in [kn-1]$ let $I_h = \{i : a_i = h\}$. Then the path $P$ corresponding to $p$ has $|I_h|$ horizontal steps at height $h$ and these are labelled in increasing left-right order with the set $I_h$.  The displacement $\disp_k(p)$ is then the number of whole squares between $P$
and the line $y=kx$.  For this reason, $\disp_k(p)$ is also known as the \emph{area} of $p$. In light of Proposition~\ref{thm:factstopark}, this explains our naming of the statistics $\arealk/\areauk$ on $\F_n^k$.

\exclude{
It has come to light that while producing this manuscript that
Nadeau \emph{et al} \cite{williams} have also produced a proof of this.  Their proof is
based on Biane's original proof for the case $k=1$, whereas we only rely
on Biane's result.  Important to us, however, is the representation
of a $k$-factorization as labelled forests:  this construction 
follows from the $k=1$ construction given \cite{irvrat} by the present
two authors;  we outline it below in Section \ref{sec:kone}.\attentionpreet{check content
of this paragraph.  I didn't read that paper}

The claim for $\bar{U}$ in Theorem \ref{thm:factstopark} follows from
the claim on $L$:  let $\theta$ be
the map on $[kn]$ that swaps $i$ and $kn -i$; apply $\theta$ to each
factor of $f$ and reverse order of the product;  the product of the
resulting cycles will be $\sigma_{[kn]}$, and its lower series will be the reverse
$\bar{U}(f)$.

We now explain why we use the terminology of area statistics for factorizations.  Let $L$ be the lattice path
beginning at $(0,0)$, then having $k$ steps to the east, followed by $k$ steps to the north,
and then repeating these east and north runs of $k$ steps until finally
terminating at $(kn, kn)$.  Given a factorization as in \eqref{eq:factform}, it can be shown that the
non-decreasing rearrangement $b_1, \ldots, b_n$ of the sequence $a_1^0, \ldots,
a_n^0$ satisfies $b_i \leq k(i-1)$ (see Section \ref{sec:parks} for details on
connections to \emph{$k$-parking functions}).  Therefore, if we draw a lattice
path $P$ from $(0,0)$ to $(kn, kn)$, having north and east steps, with $jk$ steps
east steps at height $i$, where $j = |\{t : a_t^0 = i\}|$, the lattice path lies
below the line $L$ and $\arealk(f)$ is precisely the area between $P$ and $L$.
The quantity $\areauk(f)$ can be interpreted similarly.
\attentionpreet{perhaps elaborate.  we may need an example here}

Similarly, one can make a similar interpretation of the semiareas
\attentionpreet{elaborate.  perhaps example}
}

\section{The construction for $k=1$}\label{sec:kone}

In this section we focus on the case $k=1$ of Theorem
\ref{thm:maintheorem}. The bijection $\map{\phi}{\F_n}{\R_n}$ that we  construct in this case has a particularly simple description and will be central to our analysis for general $k$.
Throughout we shall simplify our notation by omitting the value of $k$, using
$\R_n$ in place of $\R_n^1$, $\areal(f)$ for $\areal_1(f)$, \emph{etc.}  

\subsection{Arch Diagrams}
\label{sec:arch}

Let $\edgetree_n$ be the set of {vertex-rooted, edge-labelled trees} on $n$ edges, by which we mean trees with a distinguished vertex  whose edges are distinctly  labelled with $[1,n]$. Vertices are not labelled.

We typically will not distinguish between an edge and its label; \emph{i.e.} we view $[1,n]$ as the edge-set of any tree $T \in \edgetree_n$.  As with rooted forests, we regard the edges of these trees as being directed away from the root. It will be convenient to let $d(i)$ denote the child (`down') endpoint of  edge $i$.

There is a simple correspondence between $\edgetree_n$ and $\R_n$, defined as 
follows: Given $T \in \edgetree_n$, we first `push' the label of each edge $i$ 
away from the root  onto vertex $d(i)$, and then remove the root to obtain a 
rooted forest $F \in \R_n$. We call this the \emph{Cayley bijection}, written 
$\caymap\,:\,\edgetree_n \rightarrow \R_n$. See Figure \ref{fig:exampcaymapk1}.
\begin{figure}[tbp]
	\centering
	\includegraphics[width=12cm]{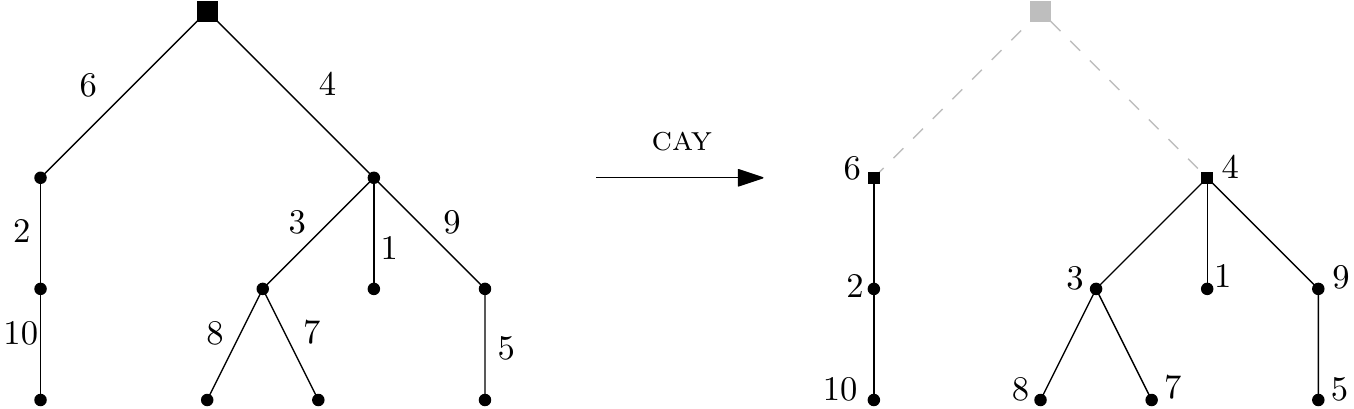}
	\caption{A vertex-rooted, edge-labelled tree 
		$T \in \edgetree_{10}$ (left) and its image $\caymap(T) \in \R_{10}$
	(right).  The boxed vertices indicate  roots.  The major index of both  $T$ and $\caymap(T)$ is 7 and their comajor index is 5.}
	\label{fig:exampcaymapk1}
\end{figure}

\exclude{
The Cayley correspondence guides our definition of hook-related statistics on
edge-labelled trees. Letting $F =\cay{T}$, we set $H(i)$=For $i \in [1,n]$, we define  the hook$H(i) := H(d(i))$ and hook length $h(i) := |H(i)|$, the left/right hooklengths
\begin{equation*}
	\he_L(i) = \sum_{j \in C_L(i)} h(j) \qquad\mathrm{ and }\qquad
	\he_R(i) = \sum_{j \in C_R(i)} h(j)
\end{equation*}
  $h_L(i) :=h_L(d(i))$
$\majl(T):=\sum_i h_L(d(i))$ etc.

\begin{equation*}
	\he_L(i) = \sum_{j \in C_L(i)} h(j) \qquad\mathrm{ and }\qquad
	\he_R(i) = \sum_{j \in C_R(i)} h(j)
\end{equation*}

}

\exclude{
major on edge-labelled trees. For  $T
\in \edgetree_n$ and $i \in [1,n]$ let $C(i)$ be the set of edges
incident with $d(i)$ excluding $i$ itself, and let $C_L(i)$ and
$C_R(i)$  be the subsets of $C(i)$ consisting of edges smaller and 
larger than $i$, respectively.  Then define
\begin{equation*}
	\he_L(i) = \sum_{j \in C_L(i)} h(j) \qquad\mathrm{ and }\qquad
	\he_R(i) = \sum_{j \in C_R(i)} h(j)
\end{equation*}
and
\begin{equation*}\label{eq:defmajedge}
	\majle(T) = \sum_{i \in [1,n]} \he_L(i) \qquad\mathrm{ and }\qquad
	\majre(T) = \sum_{i \in [1,n]} \he_R(i).
\end{equation*}
It is immediate from these definitions that 
\begin{equation}
\label{eq:edgerooted}
\majl(\cay{T}) = \majle(T) \qquad\text{and}\qquad
\majr(\cay{T}) = \majre(T)
\end{equation}
 for $T \in \edgetree_n$, as illustrated in Figure \ref{fig:exampcaymapk1}.

In what follows we will work extensively with  edge-labelled trees and ultimately rely on the Cayley bijection to translate back to the language of rooted forests. 
}

Our interest in edge-labelled trees stems from the fact that they are in natural correspondence with factorizations of full cycles.   

Consider a planar embedding of $T \in \edgetree_n$ described as follows:
\begin{enumerate}
	\item  the root is at $(0,0)$, and  all other vertices at $(i,0)$ for  $1 \leq i \leq n$;
	\item   edges are labelled $[1,n]$ and are drawn above the $x$-axis  without crossings;
		and
	\item   the sequence
		of edge labels around each vertex, taken in counterclockwise order
		beginning on the $x$-axis, is increasing.
\end{enumerate}
Following~\cite{irvrat}, we call an embedding of a member of $\edgetree_n$ satisfying (1)--(3) an \emph{arch diagram} of size $n$. Let $\archset{n}$ be set of all such diagrams up to topological equivalence. The embedding process described above provides a one-one correspondence $\edgetree_n \leftrightarrow \archset{n}$.

We  canonically label the vertices of each diagram $A \in \archset{n}$ by assigning label $i$ to the vertex at $(i, 0)$, for $0 \leq i \leq n$.  We emphasize that 
these labels are completely determined by $A$ itself, or equally by its skeletal tree $T \in \edgetree_n$. We then obtain from  $A$ a  factorization $f=(a_1\,b_1)(a_2\,b_2)\cdots(a_n\,b_n) \in
\F_n$ by letting $a_i$ and $b_i$ be the endpoints of  arch  $i$.  
See Figure~\ref{fig:exampk1}. 
\begin{figure}[tbp]
	\centering
	\includegraphics[width=.7\textwidth]{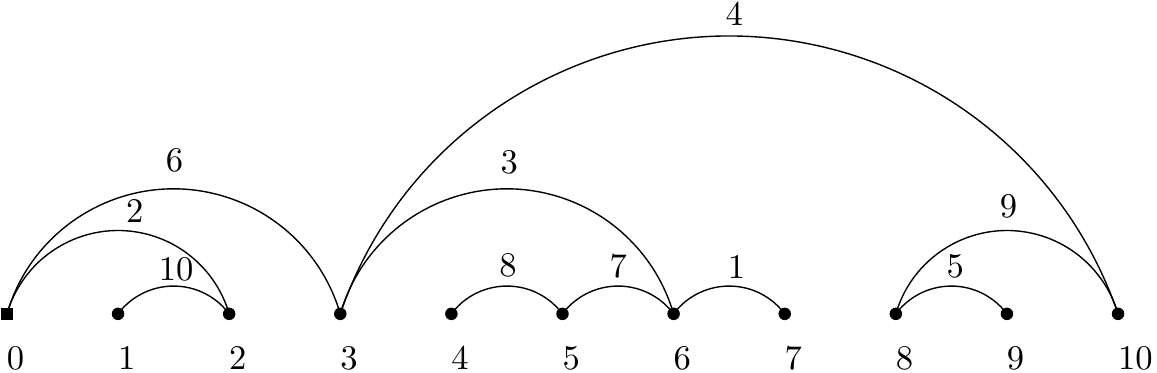}
	\caption{The arch diagram $A \in \archset{10}$ corresponding to the factorization $f = (6\, 7)(0\,
	2) (3\, 6) (3\, 10) (8\, 9) (0\, 3) (5\, 6) (4\, 5) (8\, 10)
	(1\, 2) \in \F_{10}$. Vertices are canonically  labelled.}
	\label{fig:exampk1}
\end{figure}

The transformation $A \mapsto f$ turns out to be a bijection from $\archset{n}$ to $\F_n$. We denote this function by $\map{\factmap}{\archset{n}}{\F_n}$
and its inverse by $\map{\archmap}{\F_n}{\archset{n}}$. The direct definition of $\archmap$ is obvious:  Given  $f=(a_1\,b_1)\cdots(a_n\,b_n) \in \F_n$, construct $\archmap(f)$ by drawing an arch labelled $i$ from  $(a_i, 0)$ to $(b_i, 0)$  for each $1 \leq i \leq n$.    See~\cite{irvrat} for a more detailed description of these transformations  based on the `circle-chord' construction from~\cite{gouldenyong}. The composite mapping $\factmap \circ \caymap^{-1}$  is essentially a repackaging of Moszkowski's bijection $\R_n \rightarrow \F_n$ referenced in Section~\ref{sec:factorizations}.

Let $A$ be a fixed arch diagram.  We write $l(i)$ and $r(i)$, respectively, for the left and right endpoints of edge $i$ in $A$, and we define its \emph{span}   to be the half open interval $\spann(i) := \{(x, 0) : x \in [l(i), r(i))\}$.   Since distinct edges cannot cross, their spans are either disjoint or one is contained in the other. Thus the edges of $A$ are partially ordered by inclusion of their spans.  We say $j$ \emph{covers} $i$ if $\spann(i) \subset \spann(j)$ and  there is no  arch $\ell$ with $\spann(i) \subsetneq \spann(\ell) \subsetneq \spann(j)$.

\subsection{Dual Diagrams}\label{sec:dual}

An arch diagram $A \in \archset{n}$ divides the upper half-plane $\{(x,y) \,:\, y \geq 0\}$ into $n+1$ regions, each of which contains exactly one point from the set
$H_n = \{(i + \frac{1}{2}, 0) : 0 \leq i \leq n\}$. Each arch separates two 
regions and hence two points of $H_n$.  We  construct a planar dual $D$ of $A$  
by first placing a vertex at each point of $H_n$ and then, for each $i \in 
[1,n]$,  drawing an arch labelled $\bar{i}$ between the two points of $H_n$ that 
are separated by arch $i$ of $A$.  See Figure~\ref{fig:exampdualk1}.   Note our 
use of overlined symbols $\bar{1},\bar{2},\ldots,\bar{n}$ for edge labels in the 
dual. This notational convention will be used systematically  to distinguish 
labels in $D$ from those in $A$.

So constructed, the dual of any arch diagram $A \in \archset{n}$ is an embedding of a tree $T \in \edgetree_n$ satisfying the following:
\begin{enumerate}
	\item[(1')] vertices lie at the points $H_n$, with the root at $(n+\frac{1}{2},0)$;
	\item[(2')] edges are labelled $[\bar{1},\bar{n}]$ and are drawn above the $x$-axis  without crossings;
		and
	\item[(3')]  the sequence
		of edge labels around each vertex, taken in counterclockwise order
		beginning on the $x$-axis, is \emph{decreasing}.
\end{enumerate}
Let $\dualset{n}$  be set of topologically inequivalent embeddings satisfying
$(1') - (3')$. Clearly the map
$\map{\dualmap}{\archset{n}}{\dualset{n}}$ described above is a
bijection.  We call elements of $\dualset{n}$  \emph{dual  diagrams} of
size $n$ and canonically label their  vertices by assigning label $i$ to the vertex at  $(i +
\tfrac{1}{2}, 0)$ for $0 \leq i \leq n$. 
\begin{figure}[tbp]
	\centering
	\includegraphics[width=.8\textwidth]{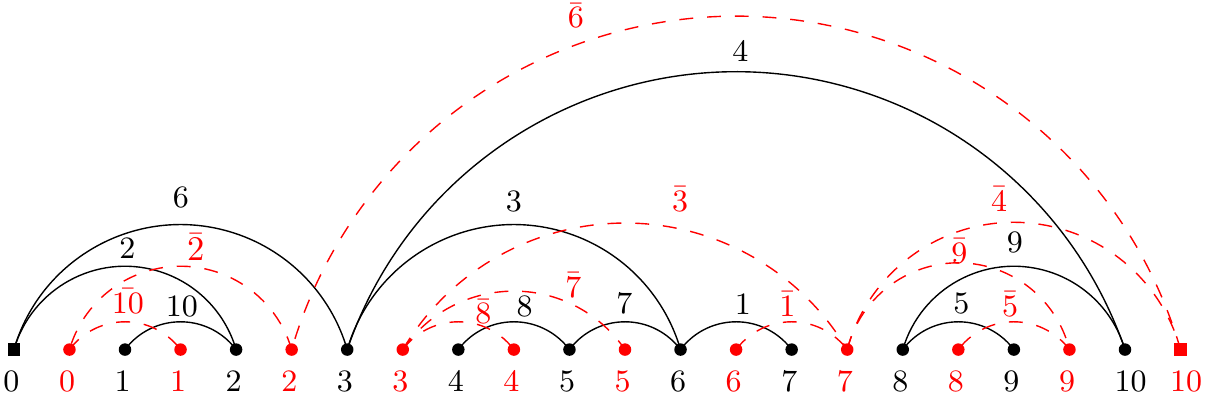}
	\caption{The arch diagram $A \in \archset{10}$ of Figure \ref{fig:exampk1} in drawn in black and its dual $D \in \dualset{10}$ in dashed red. The skeletal tree $T \in \edgetree_{10}$ of $D$  is displayed in Figure \ref{fig:exampcaymapk1} along with its corresponding forest $F \in \R_{10}$.
}
\label{fig:exampdualk1}
\end{figure}

The transparent bijection between dual diagrams and their skeletal trees allows 
us to identify $\dualset{n}$ and $\edgetree_n$. In particular, we shall view the 
Cayley bijection as a correspondence between $\dualset{n}$ and $\F_n$.  Note 
that $\cay{\dual{A}}$ is the Hasse diagram of the poset of edges of $A$ (ordered 
by inclusion).  Compare, for instance, Figures~\ref{fig:exampcaymapk1} and 
\ref{fig:exampk1}. 

\begin{lemma}
\label{lem:helpful}
Let $f=(a_1,\,b_1)\cdots(a_n\,b_n) \in \F_n$, $A =\arch{f}$, $D=\dual{A}$ and $F = \cay{D}$. For each $i \in [1,n-1]$, the following are equivalent:
\begin{enumerate}
\item  $a_i=a_{i+1}$
\item  edge $i+1$ covers edge $i$ in $A$
\item  edge $\overline{i}$ is incident with vertex $d(\overline{i+1})$ in $D$
\item  vertex $i$ is a child of vertex $i+1$ in $F$
\end{enumerate}
\end{lemma}

\begin{proof}
Condition (1) is equivalent to $l(i)=l(i+1)$ in $A$ (by definition of 
$\archmap$), and the requirement that edge labels increase counterclockwise 
around vertex $l(i)$  makes this equivalent to (2). The equivalences $(2) \iff 
(3) \iff (4)$ follow by definition of $\dualmap$ and $\caymap$.
\end{proof}



\subsection{Proof of Theorem~\ref{thm:maintheorem} (Base Case)}

\newcommand{\bari}{\bar{i}}

\begin{lemma} \label{thm:keylemma}
	Let $A \in \archset{n}, D=\dual{A}$ and $F=\cay{D}$. Then for $i \in [1,n]$ 
we have
	$h(i)=r(i)-l(i)$, where $i$ specifies a vertex of $F$ on the left-hand side 
and an edge of $A$ on the right.   More precisely, we have
	$$
		\he_L(i)=d(\bari)-l(i) \qquad\text{and}\qquad
		\he_R(i)=r(i)-d(\bari)-1.
	$$
\end{lemma}

\begin{proof}
	We refer the reader to Figure
	\ref{fig:key1}, which illustrates a portion of $A$ in solid black and $D$ in dashed red.
	\begin{figure}[tpb]
		\centering
		\includegraphics[width=15cm]{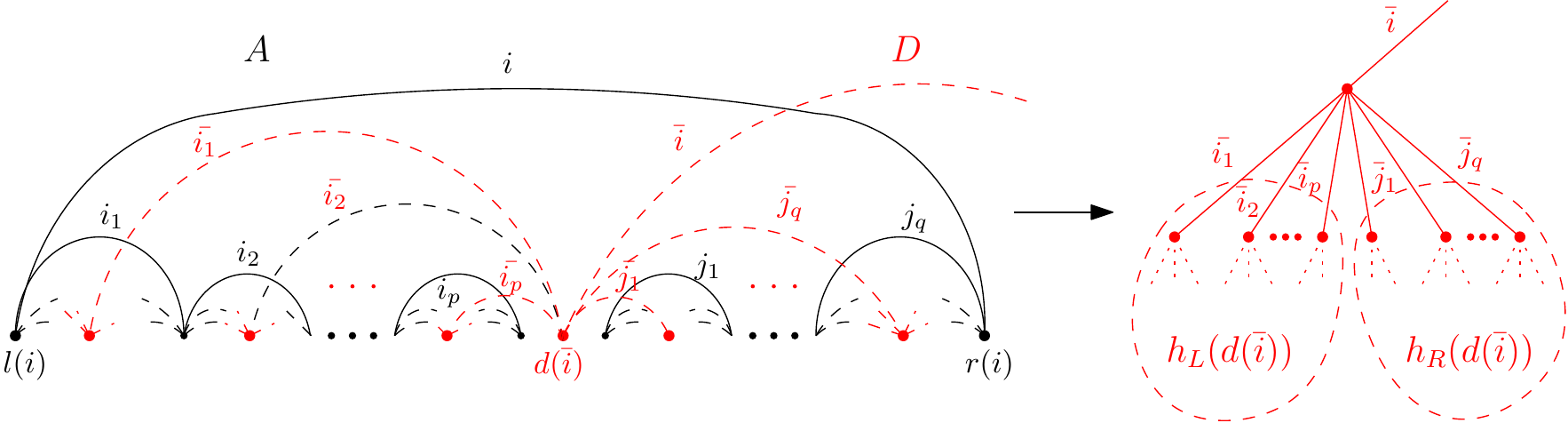}
		\caption{A portion of an arch diagram $A$ (black) and its dual $D$ (red).}
		\label{fig:key1}
	\end{figure}
	Consider edge $i$ of $A$.  
By construction,  only one dual edge  may cross this arch, namely $\bari$. Thus 
the  vertices of $D$ lying within $\spann(i)$ are precisely those in the dual 
hook $H(d(\bari))$.  Clearly there are $r(i)-l(i)$ dual vertices in $\spann(i)$, 
so we have  $h(i)=r(i)-l(i)$.		
	
Let $\bar{j}_1 > \cdots > \bar{j}_q > i > \bari_1 > \cdots > \bari_p$ be the  
edges of $D$ incident with $d(\bari)$.  Since  edge labels must decrease 
counterclockwise around $d(\bari)$, and since edges may not cross, all  vertices 
of $H(d(\bari_1)),\ldots,H(d(\bari_p))$  must lie to the left of $d(\bari)$, 
while those of $H(d(\bar{j}_1)),\ldots,H(d(\bar{j}_q))$ lie to the right. There 
are $d(\bari)-l(i)$ dual vertices in $\spann(i)$ to the left of $d(\bari)$, and 
$r(i)-d(\bari)-1$   to the right.  Thus $d(\bari)-l(i)=\sum_{s=1}^p 
h(d(\bari_s))$, and this is precisely the left hook length $h_L(i)$ in $F$.  
The expression for $h_R(i)$ follows since $h(i) = h_L(i) + h_R(i) +1$.
\end{proof}

\begin{example}\label{ex:lemma}
	In Figure \ref{fig:exampdualk1}, we have $d(\bar{4})=7$ and  hook $H(7)$ of $D$ contains all dual (red) vertices lying within $\spann(4)$, namely $\{3,4,\ldots,9\}$.  Those to the left of $7$ contribute to $\he_L(4)$ and those to the right contribute to $\he_R(4)$.  There are $d(\bar{4})-l(4)=7-3=4$ to the left and $r(4)-d(\bar{4})-1=10-7-1=2$ to the right.
\end{example}

For fixed $n \geq 1$, let $\cdamap$ denote the composite map $\caymap \circ \dualmap \circ \archmap : \F_n \rightarrow \R_n$, where again $\caymap$ is interpreted to act on the skeletal trees of dual diagrams. We now show that this bijection satisfies the conditions described by Theorem~\ref{thm:maintheorem}, thus proving the theorem in the case $k=1$. (Recall that $\mathrm{(co)semiarea}_k$ coincides with $\mathrm{(co)area}$  in this case.)
	
\begin{theorem}
	The map $\cdamap : \F_n
	\rightarrow \R_n$ is a bijection, and for $F=\cda{f}$ we have 
	$\areal(f) = \majl(F)$ and $\areau(f) = \majr(F)$.  Moreover, if
	 $f=(a_1\,b_1)\cdots(a_n\,b_n)$ then   $h(i)=a_i-b_i$ for $1 \leq i \leq n$, where $h(i)$ is the  hook length at vertex $i$ in $F$.
	\label{thm:kone}
\end{theorem}

\begin{proof}
We have already seen that $\cdamap \,:\, \F_n \rightarrow \R_n$ is a bijection.  

Let $f = (a_1\,b_1)\cdots(a_n\,b_n) \in  \F_n$, $A = \arch{f}$ and $D = \dual{A}$.  By construction  we have $l(i)=a_i$ and $r(i)=b_i$ for each arch $i$ of $A$, so Lemma~\ref{thm:keylemma} gives $h(i)=b_i-a_i$ for all $i$. 
Note that every non-root  vertex of $D$ is the endpoint $d(\bari)$ of a unique edge $\bari$. Thus $\{d(\bar{1}),\ldots,d(\bar{n})\}$ is a rearrangement of $\{0,\ldots,n-1\}$. With  Lemma~\ref{thm:keylemma}  we find that
	\begin{equation*}\label{eq:sumsegl}
		\majl(F) = \sum_{i \in [1,n]} \he_L(i) = \sum_{i \in [1,n]} (d(\bari)-l(i)) =  \binom{n}{2} - \sum_{i=1}^n a_i = \areal(f).
	\end{equation*}
	The proof that $\areau(f) = \majre(D)$ is similar. 
\end{proof}

\begin{example}\label{ex:taui}
	We return to Figure \ref{fig:exampdualk1}, which shows $A=\arch{f}$
	 and $D=\dual{A}$ for the 
	factorization $f=(a_1\,b_1)\cdots(a_{10}\,b_{10})$ given by
	$$
	 (6\, 7)(0\, 2) (3\, 6) (3\, 10) (8\, 9)
	(0\, 3) (5\, 6) (4\, 5) (8\, 10) (1\, 2).
	$$ 
	Note  $(a_4\,b_4)=(3\,\,10)$ and $h(d(\bar{4}))=7=a_4-b_4$.
 	We have $\areal(f) = \binom{10}{2}-\sum a_i = 7$ and $\areau(f)=\sum 
(b_i-1)-\binom{10}{2}=5$.  The skeletal tree of $D$ is shown in Figure
	\ref{fig:exampcaymapk1} along with its image $F$ under $\caymap$. We find that $\majl(F) = 7$ and $\majr(F) = 5$, in agreement with  Theorem
	\ref{thm:kone}.
\end{example}

\section{The main bijection for general $k$}\label{sec:genk}

In this section we prove Theorem \ref{thm:maintheorem} in the general case, building on the constructions used to prove Theorem~\ref{thm:kone}.

Geometrically, our bijection from $\F_n^k$ to $\R_n^k$ is a natural   extension
of the map $\cdamap$ described in Section~\ref{sec:kone}.  Whereas
factorizations $f \in \F_n$ correspond with  planar embeddings of edge-labelled
trees on $n$ edges  (\emph{i.e.} arch diagrams),  $k$-factorizations $f \in
\F_n^k$ correspond with   embeddings of   $k$-cacti having $n$ labelled
polygons.  We call such  configurations  \emph{generalized arch diagrams}.
Duals can be constructed analogously to the case $k=1$, and the skeletal
$k$-cacti of the duals then correspond with a $k$-forests $F \in \R_n^k$ as
described in Section~\ref{sec:kforests}. This generalized mapping $f \mapsto F$
is illustrated in Figure~\ref{fig:cactus}.
\begin{figure}[tbp]
	\centering
	\includegraphics[width=.95\textwidth]{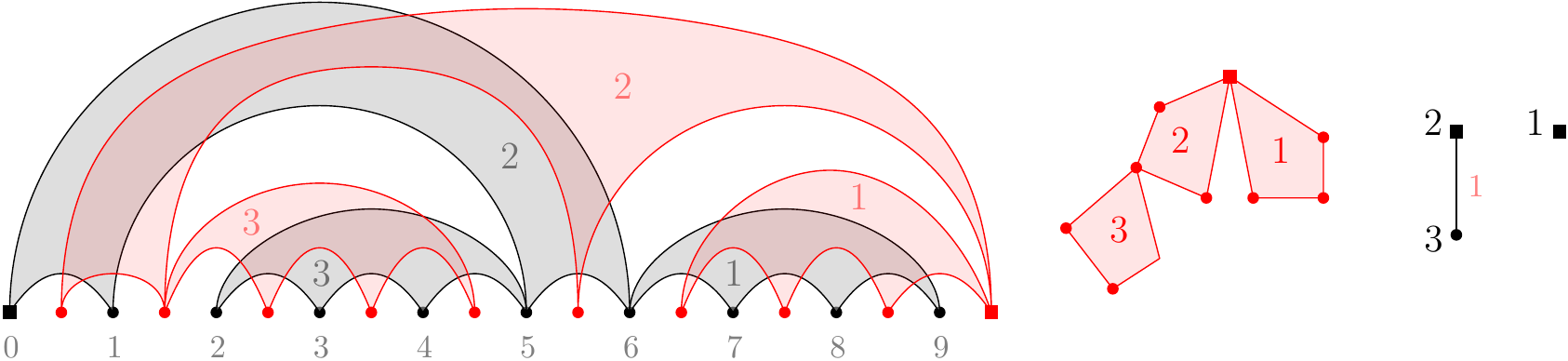}
	\caption{At left, the generalized arch diagram corresponding to the $3$-factorization $f=(6\,7\,8\,9)(0\,1\,5\,6)(2\,3\,4\,5)$ is shaded grey  and its dual is shaded  red. The skeletal $3$-cactus of the dual and its corresponding 3-forest are shown to the right. }
	\label{fig:cactus}
\end{figure}

Describing this simple geometrical transformation in sufficient detail to properly track statistics proves  somewhat cumbersome. For the purposes of bookkeeping, we shall instead define this same mapping as a chain $\F_n^k \rightarrow \F_{kn} \rightarrow \R_{kn} \rightarrow \R_n^k$ having the bijection $\cdamap:\,\F_{kn} \rightarrow \R_{kn}$  at its core.  This allows direct use of Theorem~\ref{thm:kone}.

\subsection{Defining the Bijection}

Recall from Section~\ref{sec:parks} the function $\brkmap
: \F_n^k \rightarrow \F_{kn}$  that replaces each factor of $f \in \F_n^k$ with its lower decomposition.  As noted there, $\brkmap$ is injective for all $k$ but not surjective for $k > 1$. 
Let  $\barfset_{kn} := \brkmap(\F_n^k)$ be  its image and let $\barf := \brkmap(f)$.
 
Explicitly, a $k$-factorization 
\begin{equation}\label{eq:genfact}
	f = \genkfact{k}{n} \in \F_n^k,
\end{equation} 
decomposes into the 1-factorization
$$
	\barf=(a_1^0\, a_1^1)(a_1^0\, a_1^2)\ldots(a_1^0\, a_1^k)\cdots
	(a_n^0\, a_n^1)(a_n^0\, a_n^2)\ldots(a_1^0\, a_n^k) \in \barfset_{kn}.
	$$
For notational convenience we shall label the factors of $\barf$ with a linearly ordered set different from $[1,kn]$.  We  instead use
\begin{equation}\label{eq:labsverts}
	S=\{1_1, 1_2,\ldots,1_k, 2_1, 2_2, \ldots, 2_k, \ldots, n_1, n_2, \ldots, n_k\},
\end{equation} 
ordered as they are presented above; that is, $p_q < r_s$ if and only if $p < r$ or $p=r$ and $q < s$. These labels  are applied in left-to-right order to the factors of $\barf$, so   $(a_j^0\, a_j^i)$  gets label $j_i$. 

\begin{example}\label{ex:genkbreak}
	A factorization $f \in \F_{10}^2$ and its decomposition $\barf \in \barfset_{20}$ are shown below.
	\begingroup\makeatletter\def\f@size{10}\check@mathfonts
\def\maketag@@@#1{\hbox{\m@th\large\normalfont#1}}%
	\begin{align*}
		f=\underset{1}{(0\; 1\; 4)}\underset{2}{(6\; 7\;
		8)}\underset{3}{(13\; 16\; 17)}\underset{4}{(5\; 6\;
		9)}\underset{5}{(18\; 19\; 20)}&\underset{6}{(0\; 13\;
		18)} \underset{7}{(10\; 11\; 12)}\underset{8}{(5\; 10\;
		13)}\underset{9}{(2\; 3\; 4)}\underset{10}{(14\; 15\;
		16)}\\
		\brkmap  \downarrow &\\
		\barf=
		\underset{1_1}{(0\; 1)}\underset{1_2}{(0\;
		4)}\underset{2_1}{(6\; 7)}\underset{2_2}{(6\;
		8)}\underset{3_1}{(13\; 16)}\underset{3_2}{(13\;
		17)}\underset{4_1}{(5\; 6)}&\underset{4_2}{(5\; 9)}
		\underset{5_1}{(18\; 19)}\underset{5_2}{(18\; 20)}
		\underset{6_1}{(0\; 13)}\underset{6_2}{(0\; 18)}
		\underset{7_1}{(10\; 11)}\underset{7_2}{(10\; 12)}\\
		&\underset{8_1}{(5\; 10)}\underset{8_2}{(5\; 13)}
		\underset{9_1}{(2\; 3)}\underset{9_2}{(2\; 4)}
		\underset{10_1}{(14\; 15)}\underset{10_2}{(14\; 16)}
	\end{align*}\endgroup
\end{example}

\exclude{
It follows from Lemma~\ref{lem:helpful} that  $j_{i}$ is a child of $j_{i+1}$ in $\cda{f^*}$ for all $j \in [1,n]$,  $i \in[k-1]$; and, moreover,  this property characterizes the subset $\cda{\barfset_{kn}}$ of $\R_{kn}$.  
}

Let $\barroot_{kn}:=\cda{\barfset_{kn}}$ be the image of $\F_n^k$ under $\cdamap \circ \brkmap$.  Lemma~\ref{lem:helpful} immediately provides the following characterization of this subset of  $\R_{kn}$:
\begin{equation}
	\label{eq:defbarroot}
	\barroot_{kn} =
	\{F \in \R_{kn} : 
		\text{vertex $j_{i}$ is a child of $j_{i+1}$ for all $j \in [1,n]$,  $i \in[k-1]$}\}.
\end{equation}
Certainly $\cdamap \circ \brkmap$ is a bijection from $\F_n^k$ to $\barroot_{kn}$. We ultimately want a bijection from $\F_n^k$ to $\R_n^k$, so we introduce another map $\joinmap :
\barroot_{kn} \rightarrow \R_n^k$ that acts on $\Fbar \in \barroot_{kn}$ by
\begin{equation}\label{eq:propsrnk}
	\begin{split}
		\text{\parbox{0.85\textwidth}{(1)  assigning colour $k-i$ to each edge $(j_i,r_s)$ with $r \neq j$.}} \\		
		\text{\parbox{0.85\textwidth}{(2)  merging  vertices $j_1, \dots, j_k$ into one vertex with label $j$.}}
	\end{split}
\end{equation}
Note the second step is permitted because~\eqref{eq:defbarroot} ensures  $j_1, 
\dots, j_k$ form a path, and so can be identified without creating cycles or 
parallel edges.  This function is clearly both one-one and onto. 


We have therefore defined a sequence of bijections from $\F_n^k$ to $\R_n^k$:
\begin{equation*}\label{eq:bijstring}
	\F_n^k \stackrel{\brkmap}{\longrightarrow} \barfset_{kn}
	\stackrel{\cdamap}{\longrightarrow} \barroot_{kn}
	\stackrel{\joinmap}{\longrightarrow} \R_n^k.
\end{equation*}
We denote this composite mapping by $\jcdabmap : \F_n^k
\longrightarrow \R_n^k$. In the next section we will prove that it satisfies 
the  properties of $\phi$ in Theorem~\ref{thm:maintheorem}. 

\begin{example}\label{ex:genkbreak2}
	Let $f \in \F_{10}^2$ be as in Example~\ref{ex:genkbreak}. 
	The arch diagram of $\barf$ along with its dual are shown at the top of Figure~\ref{fig:archbreakk}. The forest $\cda{\barf} \in \barroot_{20}$ is shown at bottom left, and $\jcdab{f} \in \R_{10}^2$ at bottom right.
\end{example}
\begin{figure}[tbp]
	\centering
	\includegraphics[width=\textwidth]{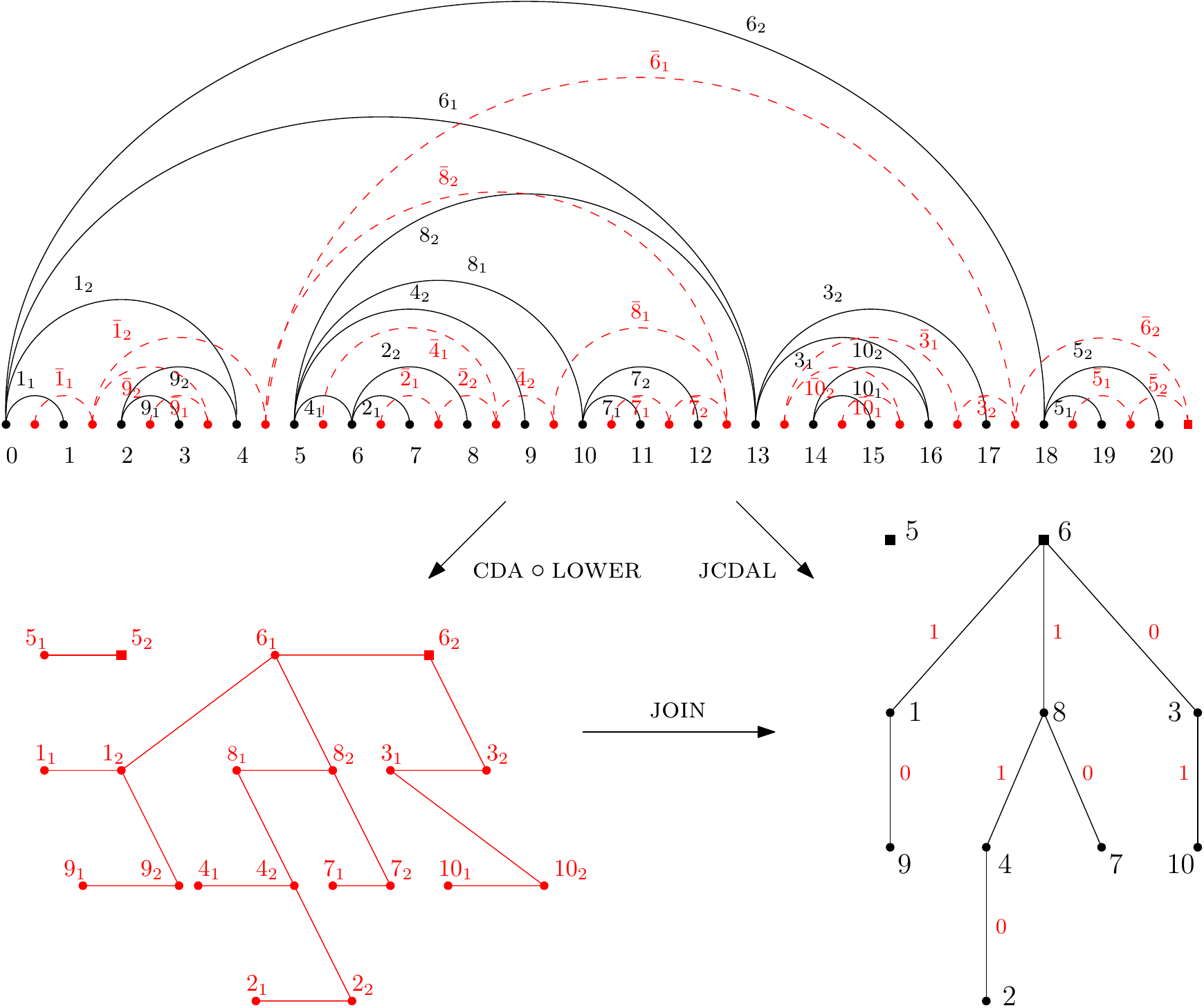}
	\caption{Let $f$ be the factorization given in Example \ref{ex:genkbreak}.
	Displayed at the top in black is $\archmap \circ \brk{f}$.  At the top in red is
$\dualmap \circ \archmap \circ \brk{f}$.  At the bottom left is $\cdamap \circ \brk{f}$
and at the bottom right is $\jcdab{f} \in \R_{10}^2$.}
	\label{fig:archbreakk}
\end{figure}

\subsection{Proof of Theorem~\ref{thm:maintheorem} (General Case)}

We  split the proof of Theorem~\ref{thm:maintheorem} into two pieces.  First we show that  $\jcdabmap$ maps $(\arealk,\sareauk)$ to $k\cdot(\majlk,\majr)$, and later that it sends $(\areauk,\sarealk) \mapsto k\cdot(\majrk,\majl)$.

\begin{proposition}
	\label{prop:mainA}
	 Let $f = \genkfact{k}{n} \in \F_n^k$ and $F = \jcdab{f} \in \R_n^k$.  Then 	$\arealk(f) = k \cdot \majlk(F)$ and $\sareauk(f) = k \cdot \majr(F)$. 
Moreover,  $a_j^k-a_j^0=k\cdot h(j)$ for $j \in [1,n]$, where $h(j)$ is the  hook length at  $j$ in $F$.\end{proposition}

\begin{proof}
Let $\barf = \brk{f}$. Then by definition of area/coarea we have
\begin{xalignat}{2}
	&\areal(\barf) = \binom{kn}{2} - k \sum_{j=1}^n a^0_j,
	&
	&\areau(\barf) = \sum_{i=1}^k \sum_{j=1}^n (a_j^i -1) - \binom{kn}{2}, \notag
\intertext{and from the definition of $\arealk$ and $\areauk$, there follows}
	\label{eq:compareareas}
	&\arealk(f) = \areal(\barf) - n \binom{k}{2},
	&
	&\sareauk(f) = \areau(\barf).
\intertext{Now let $\barF = \cda{\barf}$ and $F=\join{\barF}$.  Theorem \ref{thm:kone} gives}
	\label{eq:barfbarF}
	&\areal(\barf) = \majl(\barF),
	&
	&\areau(\barf) = \majr(\barF),
\intertext{which together with \eqref{eq:compareareas} reduces the first claims of the proposition to  the  identities}
	\label{eq:lemmaid}
		&\majl(\barF) = n \binom{k}{2} + k \cdot \majlk(F),
		&
		&\majr(\barF) = k \cdot \majr(F).
\end{xalignat}
These will be verified below in Lemma~\ref{thm:majrel}, following an example. For the final claim of the proposition,  observe that the 
factor $(a_j^0\,a_j^k)$ of $\barf$ has label $j_k$, so Theorem~\ref{thm:kone} 
implies $a_j^k-a_j^0$ is the size of the hook at vertex $j_k$ in $\barF$.   The 
definition of $\joinmap$ makes clear that this is $k$ times the size of the hook 
at vertex $j$ of $F$.
\end{proof}

\begin{example}\label{ex:barFvalues}
	Continuing with Example \ref{ex:genkbreak},  we have
		\begin{align*}
		\arealk(f) &= \binom{kn}{2} - k \cdot  \sum_j a^0_j -
		n\binom{k}{2} = 190 - 2 \cdot 73 - 10 \cdot 1 =  34\\
		 \sareauk(f) &= \sum_{i=1}^{k}
		\sum_{j=1}^n (a_j^i - 1) - \binom{kn}{2} = 202 - 190 = 12.
	\end{align*}
	We also find that  $\areal(\barf) = 44$ and $\areau(\barf) =
	12$, and for $\barF$ in Figure \ref{fig:archbreakk} (bottom left) we have $\majl(\barF) = 44$ and $\majr(\barF) = 12$, in agreement with \eqref{eq:compareareas} and \eqref{eq:barfbarF}.	
	
		The relevant statistics for $F = \jcdab{f}$ in 
	Figure \ref{fig:archbreakk} (bottom right) are
	 $\majlk(F) =\majl(F)+\chr_k(F)=8+9=17$ and $\majr(F) = 6$. Comparing these with the values of $\arealk(f)$ and $\sareauk(f)$ above verifies the first claims Proposition~\ref{prop:mainA}. Note that $a_8^k-a_0^k=13-5=8$ while $h(8)=4$ in $F$, in accord with the second claim.
\end{example}

\newcommand{\wt}{\mathrm{wt}}
\newcommand{\wthk}{\bar{h}}
\newcommand{\wtmajl}{\overline{\mathrm{maj}}}
\newcommand{\wtmajr}{\overline{\mathrm{comaj}}}

\newcommand{\wtmajlk}{\overline{\mathrm{maj}}_k}
\newcommand{\wtmajrk}{\overline{\mathrm{comaj}}_k}

\newcommand{\wtchrk}{\overline{\mathrm{chr}}_k}

To complete the proof of Proposition~\ref{prop:mainA} we introduce the notion of a \emph{weighted $k$-forest}.  This is simply a $k$-forest $F$ together with a  function $\wt:V(F) \rightarrow \{1,2,3,\ldots\}$ that assigns a positive integer weight to each vertex.   We define the \emph{weighted hook length} of a vertex $v$ in such a forest  by $\wthk(v)=\sum_{u \in H(v)} \wt(u)$.  This induces weighted versions of all statistics defined in terms of hook lengths, such as $\wthk_L(v)$, $\wthk_R(v)$, $\wtmajl(F)$, $\wtmajlk(F)$, $\wtchrk(F)$ and so on.

\begin{lemma}
	Identities~\eqref{eq:lemmaid} are valid for any 
	$\barF \in \barroot_{kn}$ and $F = \join{\barF} \in \R_n^k$.	
	\label{thm:majrel}
\end{lemma}

\begin{proof}
We begin with a general observation regarding weighted $k$-forests.
Let $W$ be such a forest with edge colouring $\col : E(W) \rightarrow 
[k-1]$, and let edge $(x,y) \in E(W)$ be such that
\begin{itemize}
\item[(A)]	 $x$ and $y$ have consecutive labels, with $x > y$, and
\item[(B)]  $\col(x,y)=0$.
\end{itemize}
Let $W'$ be the weighted $k$-forest obtained by \emph{collapsing} edge $(x,y)$ as follows: 
\begin{enumerate}
\item	 replace each edge $(y,v)$ with an edge $(x,v)$ coloured $\col(y,v)+1$,
\item	 increase $\wt(x)$ by $\wt(y)$, and
\item	 remove edge $(x,y)$ and vertex $y$.
\end{enumerate}
This  is illustrated in Figure~\ref{fig:collapsing}.  Our interest in this process arises from the fact that the $\joinmap$ map can be viewed as iteratively collapsing edges of a weighted $k$-forest.
\begin{figure}[tbp]
	\centering
	\includegraphics[width=.95\textwidth]{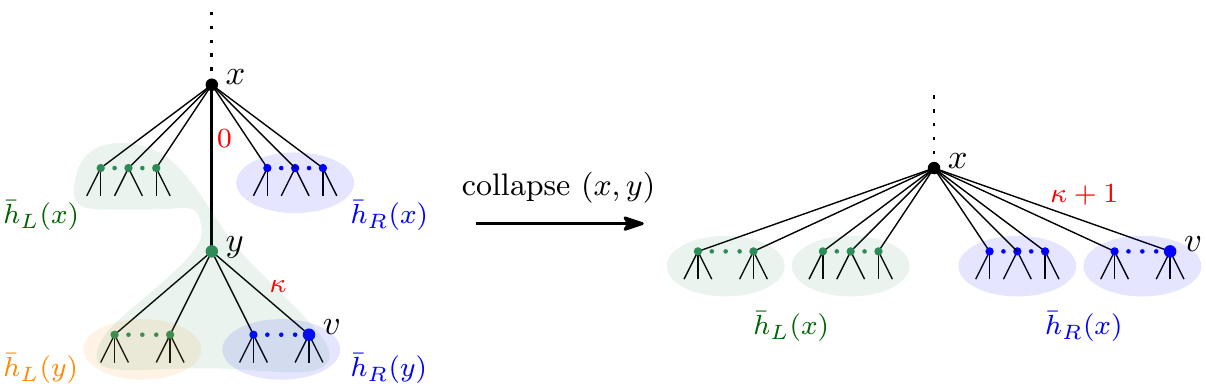}
	\caption{Collapsing edge $(x,y)$ of a weighted $k$-forest. Each edge $(y,v)$ with colour $\kappa$ is replaced by an edge $(x,v)$ with colour $\kappa+1$. The weight of vertex $x$ on the right is the sum of the weights of vertices $x$ and $y$ on the left.}
	\label{fig:collapsing}
\end{figure}

Observe that collapsing $(x,y)$ leaves $\wthk_L(u)$ and $\wthk_R(u)$ unchanged 
for all vertices $u \neq x,y$, whereas  $\wthk_L(x)$ and $\wthk_R(x)$ 
increase by $\wthk_L(y)-\wthk(y)$ and $\wthk_R(y)$, respectively, as a result 
of Condition (A). 
 The net changes to  major and comajor indices are therefore 
$(\wthk_L(y)-\wthk(y)) - \wthk_L(y)=-\wthk(y)$ and $\wthk_R(y)-\wthk_R(y)=0$, 
respectively, where the terms $-\wthk_L(y)$ and $-\wthk_R(y)$ 
account for removal of $y$.  That is,
\begin{equation}
\label{eq:weight}
\wtmajl(W')=\wtmajl(W)- \wthk(y)
\qquad\text{and}\qquad
\wtmajr(W')=\wtmajr(W).
\end{equation} 
Using Condition (B) we also have
\begin{align}
\label{eq:wtchrk}
\wtchrk(W') &= \wtchrk(W)-\underbrace{\col(x,y)}_{=0} \wthk(y)-\sum_{e=(y,v)} 
\col(e)\wthk(v) +
\sum_{e=(y,v)} (\col(e)+1)\wthk(v) \notag \\ 
 &= \wtchrk(W) +
\sum_{e=(y,v)} \wthk(v)  \\ 
&= \wtchrk(W)+\wthk(y)-\wt(y). \notag
\end{align}
Together, ~\eqref{eq:weight} and~\eqref{eq:wtchrk} give
\begin{equation}
\label{eq:iterate}
\wtmajlk(W')=\wtmajlk(W)-\wt(y).
\end{equation}

We are now prepared to prove the lemma.  Let the vertices of $\barF$ be labelled by $S$ as defined
in \eqref{eq:labsverts}.
Assign colour 0 to each edge  of $\barF$ and weight 1 to each of its vertices to transform it into a weighted $k$-forest $W^*$. Clearly $\majl(\barF)=\wtmajlk(W^*)$ and $\majr(\barF)=\wtmajr(W^*)$.

Transform $W^*$ by iteratively collapsing edges $(1_2,1_1), (1_3,1_2),\ldots,(1_k,1_{k-1})$ in that order. This yields a sequence of forests $W^*=W_1,W_2,\ldots,W_k$ such that vertex $1_i$ has weight $i$ in $W_i$.  Moreover,~\eqref{eq:weight} and~\eqref{eq:iterate} give
$$
	\wtmajr(W_k) = \wtmajr(W^*)= \majr(F^*)
$$
and
\begin{align*}
\wtmajlk(W_k)=\wtmajlk(W^*)-(1+2+\cdots+(k-1)) 
=\majl(\barF)-\binom{k}{2}
\end{align*}
Continue to iterate by collapsing edges $(j_2,j_1), (j_3,j_2),\ldots,(j_k,j_{k-1})$, for each $j=2,\ldots,n$. 
When complete, this results in a weighted $k$-forest ${W}$ on vertices $\{1_k,2_k,\ldots,j_k\}$ such that
\begin{itemize}
\item 	every vertex has weight $k$,
\item	$\wtmajlk(W)=\majl(\barF)-n\binom{k}{2}$, and
\item	$\wtmajr(W)=\majr(\barF)$.
\end{itemize}
  Clearly ${W}$ is precisely the $k$-forest $F=\joinmap(\barF)$ but with vertices weighted $k$, so we have $\wtmajlk({W})=k \cdot \majlk(F)$ and $\wtmajr(W)=k\cdot\majr(F)$. Therefore  $\majl(\barF)=k\cdot \majlk(F)+n\binom{k}{2}$ and $\majr(\barF)=k\cdot\majr(F)$, as desired.
\end{proof}

Finally, we  prove the remainder of Theorem \ref{thm:maintheorem}.

\begin{proposition}
	Let $f = \genkfact{k}{n} \in \F_n^k$, and $F = \jcdab{f}$.
	Then $\sarealk(f) = k\cdot  \majl(F)$ and
	$\areauk(f) = k \cdot \majrk(F)$.
	\label{prop:mainB}
\end{proposition}

\begin{proof}
	The proof is similar to that of Proposition \ref{prop:mainA}, with
	one modification:  we use a different though analogous
	intermediate function $\hatbrkmap : \F_n^k \rightarrow
	\R_{kn}$ defined by replacing each factor of a $k$-factorization with its \emph{upper} decomposition; see~\eqref{eq:upperdecomp}.  Thus if $f = \genkfact{k}{n} \in \F_n^k$,
	then $\hatf := \hatbrk{f}$ is obtained by expanding the factor $(a_j^0\, \cdots a_j^k)$ into  $(a_j^0\, a_j^k) (a_j^1\, a_j^k)
	\cdots (a_j^{k-1}\, a_j^k)$.  We label the factors
	of $\hatf$ with the ordered set $S$ of
	\eqref{eq:labsverts}, assigning factor $(a_j^i\; a_j^k)$ the label
	$j_{i+1}$, and this labelling propagates  to the resulting
	forests.
	  We immediately obtain the
	relationship
	\begin{equation}\label{eq:areaanal}
		\areauk(f) = \areau(\hatf) - n \cdot \binom{k}{2}\;\;\;
		\textnormal{ and }\;\;\; \sarealk(f) = \areal(\hatf),	
	\end{equation}
	which is analogous to \eqref{eq:compareareas}.  
	Clearly $\hatbrkmap$, like $\brkmap$,
	is injective.  Let $\hatroot_{kn}=\cdamap \circ \hatbrkmap(\hatfset_{kn})$.  Through the obvious analogue of Lemma~\ref{lem:helpful}, this subset of $\R_{kn}$ is seen to be characterized by
	\begin{equation*}
		\hatroot_{kn}= \{F \in \R_n^k\,:\, 
			\text{vertex $j_{{i+1}}$ is a child of $j_{i}$ for all $j \in [1,n]$,  $i \in[k-1]$}\}.
		\label{eq:defhatroot}
	\end{equation*}
	Setting $\hatF = \cda{\hatf}$, we know from
	Theorem \ref{thm:kone} that
	\begin{equation}\label{eq:areaforestanal}
		\areau(\hatf) = \majr(\hatF)\qquad  \textnormal{
		and }\qquad \areal(\hatf) = \majl(\hatF).
	\end{equation}
	We define  $\joinmap : \hatroot_{kn} \rightarrow
	\R_n^k$ as before, and similar to 
	 Lemma \ref{thm:majrel} we find for $F = \joinmap(\hatF)$ that
	\begin{equation}\label{eq:hatforestanal}
		\majr(\hatF) = n \cdot \binom{k}{2} + k \cdot
		\majrk(F)\qquad \textnormal{ and }\qquad \majl(\hatF) = k
		\cdot \majl(F).
	\end{equation}
	Finally, it is easy to see that the composition
	$\joinmap \circ \cdamap \circ \hatbrkmap : \F_n^k
	\longrightarrow \R_n^k$ is precisely 
$\jcdabmap : \F_n^k
	\longrightarrow \R_n^k$.  Combining \eqref{eq:areaanal},
	\eqref{eq:areaforestanal} and \eqref{eq:hatforestanal} completes the proof.
\end{proof}

\begin{example}
	Let $f \in \F_{10}^2$ be as in Example
	\ref{ex:genkbreak}. Then  $\hatf = \hatbrk{f}$ is
		\begingroup\makeatletter\def\f@size{10}\check@mathfonts
\def\maketag@@@#1{\hbox{\m@th\large\normalfont#1}}%
	\begin{align*}
		\hatf=
		\underset{1_1}{(0\; 4)}\underset{1_2}{(1\;
		4)}\underset{2_1}{(6\; 8)}\underset{2_2}{(7\;
		8)}\underset{3_1}{(13\; 17)}\underset{3_2}{(16\;
		17)}\underset{4_1}{(5\; 9)}&\underset{4_2}{(6\; 9)}
		\underset{5_1}{(18\; 20)}\underset{5_2}{(19\; 20)}
		\underset{6_1}{(0\; 18)}\underset{6_2}{(13\; 18)}
		\underset{7_1}{(10\; 12)}\underset{7_2}{(11\; 12)}\\
		&\underset{8_1}{(5\; 13)}\underset{8_2}{(10\; 13)}
		\underset{9_1}{(2\; 4)}\underset{9_2}{(3\; 4)}
		\underset{10_1}{(14\; 16)}\underset{10_2}{(15\; 16)}.
	\end{align*} \endgroup
	The arch diagram of $\hatf$ and its dual are given in Figure
	\ref{fig:genkbreakhat}. We see that 
	\begin{equation}\label{eq:compareas2}
\begin{aligned}
	\areauk(f) &= k \sum_{j=1}^n (a_j^k - 1) - \binom{kn}{2} - n
	\cdot \binom{k}{2}  = 222 - 190 - 10 = 22,\\
	 \areau(\hatf) &= 32,
\end{aligned}
\end{equation}
confirming the first half of
\eqref{eq:areaanal};  also,
\begin{equation}\label{eq:compsemiarea}
	\sarealk(f) = \binom{kn}{2} - \sum_{i=0}^{k-1} \sum_{j=1}^n a_j^i = 190 -
	174 = 16 = \areal(\hatf),
\end{equation}
confirming the second half. Figure
\ref{fig:genkbreakhat} also shows $\hatF = \cda{\hatf}$ and
$F = \jcdab{f}$.  There we find  
$\majl(\hatF) = 16$, $\majr(\hatF) =32$,
$\majl(F) = 8$, and $\majrk(F) = \majr(F) +
\cochr_k(F) = 6+5=11$.
Comparison with \eqref{eq:compareas2} and
\eqref{eq:compsemiarea} confirms \eqref{eq:areaforestanal} and  \eqref{eq:hatforestanal}. 
\end{example}
\begin{figure}[tbp]
	\centering
	\includegraphics[width=15cm]{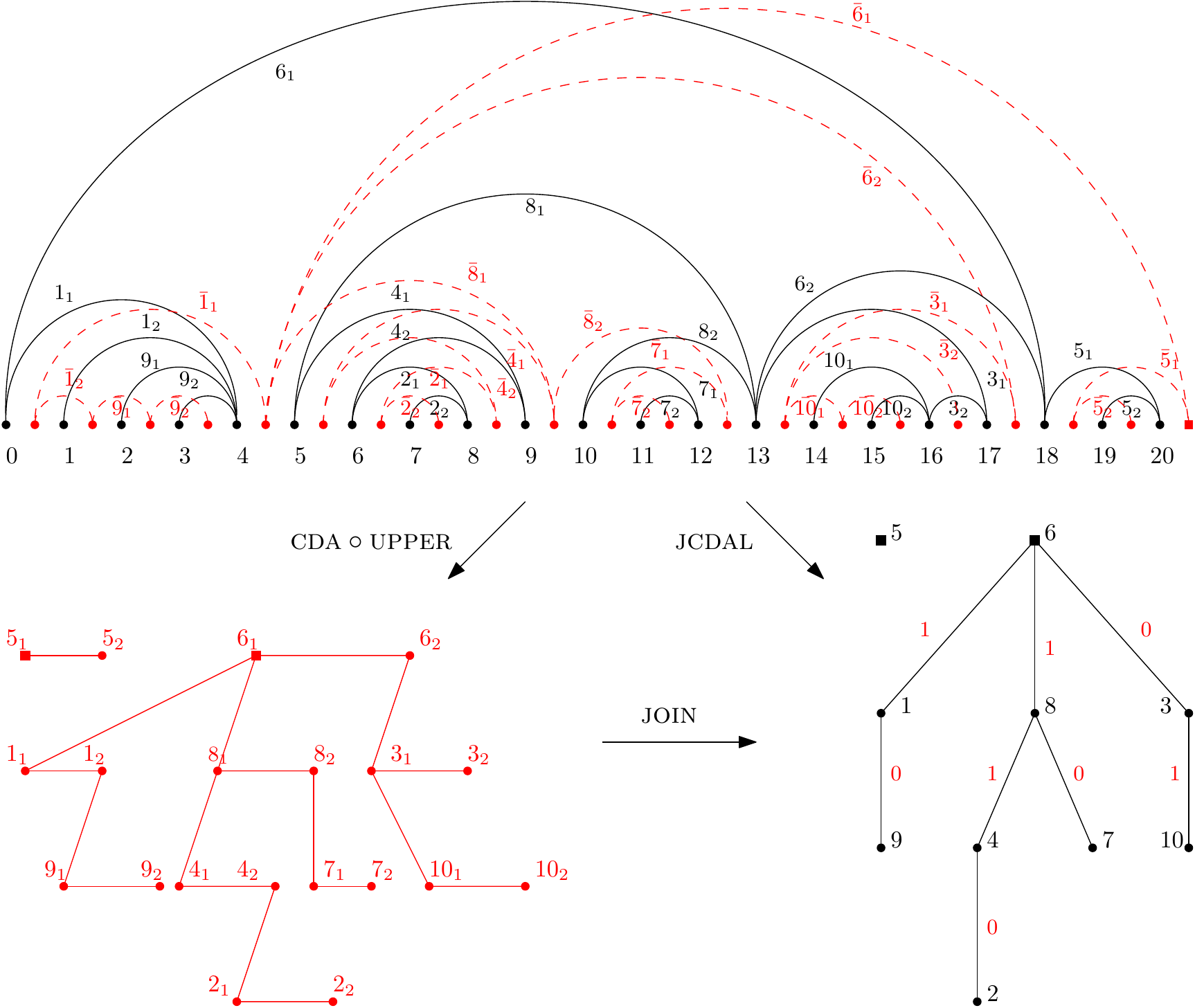}
	\caption{The arch diagram and dual of $\hatf = \hatbrk{f}$,
where $f$ is given in Example \ref{ex:genkbreak}.  Also shown are
$\hatF=\cda{\hatf}$ (bottom left) and $F = \joinmap(\hatF)=\jcdab{f}$ (bottom right).}
	\label{fig:genkbreakhat}
\end{figure}
\bibliographystyle{amsalpha}
\bibliography{kparking}
\end{document}